\documentclass[12pt]{amsart}
\usepackage{graphicx}
\usepackage{amsmath}
\usepackage{amsfonts}

\newtheorem{thm}{Theorem}[section]

\newtheorem{lem}[thm]{Lemma}
\newtheorem{prop}[thm]{Proposition}
\theoremstyle{definition}
\newtheorem{defn}[thm]{Definition}
\theoremstyle{remark}
\newtheorem{rem}[thm]{Remark}
\numberwithin{equation}{section}
\newcommand{\Real}{\mathbb R}
\newcommand{\Nat}{\mathbb N}

\newcommand{\supp}{\rm supp}
\begin{document}

\title[]{Semialgebraic Calder{\'o}n-Zygmund theorem on regularization of the distance function}%
\author{Beata Kocel-Cynk, Wies{\l}aw Paw{\l}ucki and Anna Valette}%

\address{(B. Kocel-Cynk) Instytut Matematyki Politechniki Krakowskiej, ul. Warszawska 24, 31-155 Kraków, Poland}%
\email{beata.kocel-cynk{@}pk.edu.pl}%

\address{(W. Pawłucki) Instytut Matematyki Uniwersytetu Jagiellońskiego, ul. Prof. St. Łojasiewicza 6, 30-048 Kraków, Poland}%
\email{wieslaw.pawlucki{@}im.uj.edu.pl}%

\address{(A. Valette) Katedra Teorii Optymalizacji i Sterowania Uniwersytetu Jagiellońskiego, ul. Prof. St. Łojasiewicza 6, 30-048 Kraków, Poland}%
\email{anna.valette{@}uj.edu.pl}%

\subjclass{Primary 14P20. Secondary 57R35, 14P10, 32B20}%
\keywords{Semialgebraic set, semialgebraic mapping, Nash appro\-xi\-mation, $\Lambda_p$-regular mapping}%

\date{28 November, 2023
$\phantom{......................................................................................................................}$
Anna Valette was partially supported by Narodowe Centrum Nauki (Poland); grant no. 2021/43/B/ST1/02359}
% ----------------------------------------------------------------
\begin{abstract}
We prove that, for any closed semialgebraic subset $W$ of $\Real^n$ and for any positive integer $p$, there exists a Nash function $f:\Real^n\setminus W\longrightarrow (0, \infty)$ which is equivalent to the distance function from $W$ and at the same time it is $\Lambda_p$-regular in the sense that $|D^\alpha f(x)|\leq C d(x, W)^{1- |\alpha|}$, for each $x\in \Real^n\setminus W$ and each $\alpha\in \Nat^n$ such that $1\leq |\alpha|\leq p$, where $C$ is a positive constant. In particular, $f$ is Lipschitz. Some applications of this result are given.
\end{abstract}
\maketitle
% ----------------------------------------------------------------
\section{Introduction}

The Calderón-Zygmund theorem on regularization of the distance function asserts that for any closed subset $W\subset \Real^n$ there exists a $\mathcal C^\infty$-function $f: \Real^n\setminus W\longrightarrow (0, \infty)$ equivalent to the distance function from $W$; i.e. there exists a constant $A > 0$ such that
$$
A^{-1}d(x, W)\leq f(x) \leq Ad(x, W), \quad\text{for each $x\in \Real^n\setminus W$}
$$
and, moreover, there are constants $B_{\alpha} > 0$ \, $(\alpha\in \Nat^n)$, \, such that
$$
|D^\alpha f(x)|\leq B_\alpha(d(x, W))^{1-|\alpha|}, \quad\text{for each $x\in \Real^n\setminus W$ and each $\alpha\in \Nat^n$}.
$$
It was introduced in connection with a study of elliptic partial differential equations (cf. [2]) and appears a useful tool in analysis (cf. [14, Chapter VI]).

\bigskip
Since semialgebraic geometry (cf. [1]) together with its generalizations (subanalytic geometry (cf.[13]), o-minimal geometry (cf.[3])) appears very valuable in areas of applied mathematics such as robotics and CAD, it was an interesting open question if the Calderón-Zygmund theorem has a counterpart in the semialgebraic category. Our aim is to give a positive answer; namely, we prove the following.

\begin{thm}
For any closed semialgebraic subset $W$ of $\Real^n$ and any positive integer $p$, there exists a Nash function $($i.e. semialgebraic and $\mathcal C^\infty$$)$
$f:\Real^n\setminus W\longrightarrow (0, \infty)$ and positive constants $A, B$ such that, for each $x\in \Real^n\setminus W$

\begin{equation}
A^{-1}d(x, W)\leq f(x) \leq Ad(x, W),
\end{equation}
and
\begin{equation}
|D^\alpha f(x)|\leq B(d(x, W))^{1-|\alpha|}, \, \text{where $\alpha\in\Nat^n$ and $|\alpha|\leq p$}.
\end{equation}
\end{thm}

\bigskip
The proof of Theorem 1.1 is based on $\Lambda_p$-regular stratifications (see Section 2) introduced by the second author with Krzysztof Kurdyka in [7], in connection with a subanalytic version of the Whitney extension theorem, combined with a version of the Efroymson-Shiota approximation theorem, cited below (see Theorem 1.4). In fact, we will need the following generalization of the notion of $\Lambda_p$-regular function considered in [7].

\begin{defn}
Let $W\subset \Real^n$ be a closed semialgebraic subset, let $p, k\in \mathbb Z$, where $p > 0$. Let $\varOmega\subset \Real^n$ be an open semialgebraic subset disjoint from $W$. We say that a semialgebraic $\mathcal C^p$-function $f:\varOmega \longrightarrow \Real$ is $\Lambda^k_p(W)$-\emph{regular} if there exists a constant $M > 0$ such that
$$
|D^\alpha f(x)|\leq Md(x, W)^{k-|\alpha|},
$$

for each $x\in \varOmega$ and $\alpha\in \Nat^n$ such that $1\leq |\alpha|\leq p$.
\medskip

When $f$ is $\Lambda^1_p(\partial \varOmega)$-regular we say that $f$ is $\Lambda_p$-\emph{regular} (as in [7]).
\end{defn}

\bigskip
Our main effort in this paper is focused on proving the following approximation theorem for Lipschitz functions (notice that the distance function is a particular case).

\bigskip
\begin{thm}
Let $W\subset \Real^n$ be any closed semialgebraic subset and let $p$ be a positive integer. Let $g:\Real^n\longrightarrow \Real$ be any semialgebraic Lipschitz function vanishing on $W$. \newline Then, for any $\kappa > 0$,  there exists a $\Lambda^1_p(W)$-regular function $f:\Real^n\setminus W\longrightarrow \Real$ such that, for each $x\in \Real^n\setminus W$,
$$
|f(x) - g(x)|\leq \kappa d(x, W).
$$
\end{thm}

\bigskip
The proof of Theorem 1.3 is based on a special $\Lambda^0_p(W)$ partition of unity, which we establish in Section 3, and we believe has its own interest. It is also worth noting that Theorem 1.3 with the given proof holds true in the setting of any o-minimal structure on the field of real numbers $\Real$.

\bigskip
We will use a special case of the Efroymson-Shiota approximation theorem. To quote this, we first recall the definition of the \emph{semialgebraic $\mathcal C^p$-topology}. Let $G$ and $H$ be open semialgebraic subsets in $\Real^n$ and in $\Real^m$, respectively. Let $p$ be a non-negative integer.
Denote by $\mathcal N^p(G, H)$ the set of all semialgebraic $\mathcal C^p$-mappings from $G$ to $H$; i.e. $\mathcal C^p$-mappings with semialgebraic
graphs. Let $f\in \mathcal N^p(G, H)$. Then basic neighborhoods of $f$ in $\mathcal N^p(G, H)$ in the semialgebraic $\mathcal C^p$-topology are of the form
$$
U_\varepsilon(f) = \{h\in \mathcal N^p(G, H): \, |D^\alpha f(x) - D^\alpha h(x)|\leq \varepsilon(x), \phantom{..................}
$$
$\phantom{........................................}$ whenever $\alpha\in \Nat^n$, $|\alpha|\leq p$ and $x\in G$\},
$\phantom{...........................................................................................................................}$
where $\varepsilon: G\longrightarrow (0, \infty)$ is any semialgebraic positive continuous fun\-ction on $G$.

\medskip
\begin{thm}[Efroymson-Shiota approximation theorem]
$\phantom{...................................................................................................................}$

\medskip
Nash mappings $($i.e. semialgebraic and $\mathcal C^\infty$$)$ from $G$ to $H$ are dense in $\mathcal N^p(G, H)$ in the semialgebraic $\mathcal C^p$-topology.
\end{thm}

\medskip
This deep result originating in the paper of Efroymson [4] for $p = 0$ (compare also [11]), was completed and generalized, for any non-negative $p$, by Shiota in [12]. In fact, Shiota's formulation is stronger than Theorem 1.4; namely, the sets $G$ and $H$ above can be any Nash submanifolds embedded in $\Real^n$ and in $\Real^m$, respectively.

\medskip
It is now a simple matter to see that Theorem 1.1 is a consequence of first applying Theorem 1.3, where we put $g(x) = d(x, W)$, followed by Theorem 1.4, applied to a resulting $f$. Hence, the rest of our paper is devoted to proving Theorem 1.3.

\bigskip
\section{$\Lambda_p$-regular cells}

We recall after [7] (see also [8], [9] and [10]), the definition of $\Lambda_p$-regular cells in $\Real^n$.

\medskip
\begin{defn}
$\phantom{.......................................................................................................................}$

\medskip
Let $p$ be a positive integer. We say that $S$ is an \emph{open $\Lambda_p$-regular cell} in $\Real^n$ if
\begin{equation}
\text{$S$ is any open interval in $\Real$, when $n = 1$;}
\end{equation}
\begin{equation}
S = \{(x', x_n): x'\in T, \psi_1(x') < x_n < \psi_2(x')\},
\end{equation}
where $x' = (x_1,\dots, x_{n-1})$, $T$ is an open $\Lambda_p$-regular cell in $\Real^{n-1}$ and every $\psi_i$ \, $(i\in \{1, 2\})$ \, is either a semialgebraic $\Lambda_p$-regular function on $T$ (see Definition 1.2) with values in $\Real$, or identically equal to $-\infty$, or identically equal to $+\infty$, and $\psi_1(x') < \psi_2(x')$, for each $x'\in T$, when $n > 1$.

\medskip
Extending the above definition, we say that $S$ is an \emph{$m$-dimensional $\Lambda_p$-regular cell} in $\Real^n$, where $m\in \{0,\dots, n-1\}$, if
\begin{equation}
S = \{(u, w): \, u\in T, w = \varphi(u)\},
\end{equation}
where $u = (x_1,\dots, x_m)$, $w = (x_{m+1},\dots, x_n)$, \, $T$ is an open $\Lambda_p$-regular cell in $\Real^m$, and $\varphi: T\longrightarrow \Real^{n-m}$ is a semialgebraic $\Lambda_p$-regular mapping.
\end{defn}

\medskip
\begin{rem} One easily checks by induction that every $\Lambda_p$-regular cell is Lipschitz in the sense that each of the functions $\psi_i$ in (2.2), if finite, as well as the mapping $\varphi$ in (2.3), are Lipschitz. Besides, every $\Lambda_p$-regular cell in $\Real^n$ is a semialgebraic connected $\mathcal C^p$-submanifold of $\Real^n$.
\end{rem}

\begin{defn}
$\phantom{...........................................................................................................................}$

\medskip
Let us recall that a (\emph{semialgebraic}) $\mathcal C^p$-\emph{stratification} of a (semialgebraic) subset $E$ of $\Real^n$ is a finite decomposition $\mathcal S$ of $E$ into (semialgebraic) connected $\mathcal C^p$-submanifolds of $\Real^n$, called \emph{strata}, such that for each stratum $S\in \mathcal S$, its \emph{boundary} in $E$; i.e. $\partial_ES := (\overline{S}\setminus S)\cap E$ is the union of some strata of dimensions $< \dim S$. If $A_1,\dots, A_k$ \, $(k\in \Nat)$ \, are subsets of $E$, we call a stratification $\mathcal S$ \emph{compatible with the subsets $A_1,\dots, A_k$}, if each $A_j$ is a union of some strata.
\end{defn}

\medskip
The following proposition is crucial in the proof of Theorem 2.6 below, which is a fundamental theorem on $\Lambda_p$-stratifications.

\medskip
\begin{prop}[\lbrack 8, Corollary to Proposition 4 \rbrack]
Let $\varPhi:\varOmega\longrightarrow \Real$ be a semialgebraic $\mathcal C^1$-function defined on a semialgebraic open subset $\varOmega$ of $\Real^n$ such that
$$
\Big|\frac{\partial \varPhi}{\partial x_j}\Big|\leq M \qquad (j\in \{1,\dots, n\}),
$$
where $M$ is a positive constant, and let $p$ be a positive integer. Then there exists a closed semialgebraic nowhere dense subset $Z$ of $\varOmega$ such that $\Phi$ is of class $\mathcal C^p$ on $\varOmega\setminus Z$ and
$$
|D^\alpha \varPhi(u)|\leq C(n,p)Md(u, Z\cup\partial \Omega),
$$
whenever $u\in \varOmega\setminus Z$, $\alpha\in \Nat^n$, $1\leq|\alpha|\leq p$, and where $C(n, p)$ is a positive integer depending only on $n$ and $p$.
\end{prop}

\medskip
\begin{rem}
If $\varPhi:\varOmega \longrightarrow \Real$ is a semialgebraic Lipschitz function with a constant $M$ defined on a semialgebraic open subset $\varOmega$ of $\Real^n$, then there exists a closed semialgebraic nowhere dense subset $Z'$ of $\varOmega$ such that $\varphi$ is of class $\mathcal C^1$ on $\varOmega\setminus Z'$ and
$$
\Big|\frac{\partial \varPhi}{\partial x_j}\Big|\leq M \quad\text{on}\quad \varOmega\setminus Z'.
$$
\end{rem}

\begin{thm}
Let $p$ be a positive integer. Given any finite number $A_1,\dots, A_k$ of semialgebraic subsets of a semialgebraic subset $E$ of $\Real^n$, and a semialgebraic Lipschitz mapping
$g: E\longrightarrow \Real^d$, where $d\in \Nat$, there exists a semialgebraic $\mathcal C^p$-stratification $\mathcal S$ of $E$ compatible with sets $A_1,\dots, A_k$ and such that every stratum $S\in \mathcal S$, after an orthogonal linear change of coordinates\footnote{By [10], permutations of coordinates $x_1,\dots, x_n$ suffice.}  in $\Real^n$, is a $\Lambda_p$-regular cell in $\Real^n$ and if $S$ is open then $g|S$ is $\Lambda_p$-regular, while in the case $\dim S = m < n$, when $S$ is of the form $(2.3)$,
\begin{equation}
\text{the mapping $T\ni u\longmapsto g(u, \varphi(u))\in \Real^d$ is $\Lambda_p$-regular.}
\end{equation}
\end{thm}

\begin{proof}
The proof follows the inductive procedure as in the proof of Proposition 4 in [7] (or that of Theorem 3 in [8]); i.e. the induction on $\dim E$. The only difference is that, at each inductive step, constructing strata of dimension $< m$, we have to take into account the Lipschitz mapping $g$ restricted to strata of dimension $m$ making use of Remark 2.5 and Proposition 2.4.
\end{proof}

\bigskip
\section{A partition of unity}

\begin{defn}
$\phantom{.............................................................................................................................................}$

\medskip
Let $W$ be a closed semialgebraic subset of $\Real^n$ and let $Z\subset \Real^n\setminus W$. We will consider the following open neighborhoods of $Z$ in $\Real^n$

$$
G_\eta(Z, W) :=\{x\in \Real^n\setminus W: \, d(x, Z) < \eta d(x, W)\},
$$
where $\eta > 0$. (We adopt the convention that $d(x, \emptyset) = \infty$.)
\end{defn}

\medskip
The main result of this section is the following theorem on $\Lambda^0_p(W)$-partition of unity, which can be considered as a semialgebraic counterpart of the famous Whitney partition of unity.

\begin{thm}
$\phantom{................................................................................................................................................}$

\medskip
Let $W$ be a closed semialgebraic subset of $\Real^n$ and let $U_1,\dots, U_s$ be any finite covering of $\Real^n\setminus W$ by semialgebraic subsets.
Then, for any positive integer $p$ and any $\eta > 0$, there exist $\Lambda^0_p(W)$-regular functions $\omega_i: \Real^n\setminus W\longrightarrow [0, 1]$ \,\, $(i\in \{1,\dots, s\})$ \, such that \, $\omega_1 + \dots + \omega_s \equiv 1$ \, on $\Real^n\setminus W$ \, and \, $\supp$$\omega_i \subset G_{\eta}(U_i, W)$ \, $(i\in \{1,\dots, s\})$, where $\supp$$\omega_i$ denotes the closure of $\{x\in \Real^n\setminus W: \, \omega_i(x)\neq 0\}$ in $\Real^n\setminus W$.
\end{thm}

Before starting the proof of Theorem 3.2, we will prove the following two simple lemmas.

\begin{lem}

Let $W$ be a closed  semialgebraic subset of $\Real^n$ and let $\varOmega$ be an open semialgebraic subset of $\Real^n$ disjoint from $W$. If $f:\varOmega\longrightarrow \Real$ is a $\Lambda^k_p(W)$-regular function and $g:\varOmega\longrightarrow \Real$ is a $\Lambda^l_p(W)$-regular function, where $k, l, p\in \mathbb Z$, $p > 0$ and if there exists $A > 0$ such that $|f(x)|\leq Ad(x, W)^k$ and $|g(x)|\leq Ad(x, W)^l$, for each $x\in \varOmega$, then the function $fg$ is $\Lambda^{k+l}_p(W)$-regular.
\end{lem}

\begin{proof} Directly from the Leibnitz formula.
\end{proof}

\medskip
\begin{lem}

Let $W$ be a closed  semialgebraic subset of $\Real^n$ and let $\varOmega$ be an open semialgebraic subset of $\Real^n$ disjoint from $W$. If $f:\varOmega\longrightarrow \Real$ is a $\Lambda^k_p(W)$-regular function, where $k, p\in \mathbb Z$, $p > 0$, and there exists $a > 0$ such that $ad(x, W)^k \leq |f(x)|$, for each $x\in \varOmega$, then the function $1/f$ is $\Lambda^{-k}_p(W)$-regular.

\end{lem}

\begin{proof} Observe that $D^\alpha(1/f)$ \, $(\alpha\in \Nat^n \setminus \{0\})$, is a linear combination, with integral coefficients independent of $f$, of products of the form
$$
f^{-(m + 1)}(D^{\beta_1}f)\dots (D^{\beta_m}f),
$$
where $1\leq m \leq |\alpha|$, $\beta_1,\dots, \beta_m\in \Nat^n\setminus \{0\}$, and $\sum_{i=1}^m \beta_i = \alpha$. Hence, we get
$$
|D^\alpha(1/f)(x))|\leq C d(x, W)^{-k(m+1)}d(x, W)^{k-|\beta_1|}\dots d(x, W)^{k - |\beta_m|} =
$$
$$
Cd(x, W)^{-k-|\alpha|},
$$
where $C > 0$.
\end{proof}

\medskip
\begin{lem}
Let $W$ be a closed  semialgebraic subset of $\Real^n$, let $\varOmega$ be an open semialgebraic subset of $\Real^n$ disjoint from $W$ and let $p$ be a positive integer. If $f:\varOmega\longrightarrow \Real$ is a bounded $\Lambda^0_p(W)$-regular function and $\varPhi: \Real\longrightarrow \Real$ is a semialgebraic $\mathcal C^p$-function, then $\varPhi\circ f$ is a $\Lambda^0_p(W)$-regular function.
\end{lem}

\begin{proof}

Observe that if $\alpha\in \Nat^n$ and $1\leq |\alpha|\leq p$, then $D^\alpha(\varPhi\circ f)$ can be represented as a linear combination, with integral coefficients independent of $\varPhi$ and $f$, of products of the form
$$
\varPhi^{(i)}(f)(D^{\beta_1}f) \dots (D^{\beta_i}f),
$$
where $1\leq i \leq |\alpha|$  and  $\beta_1,\dots, \beta_i\in \Nat^n\setminus \{0\}$ are such that $\beta_1 + \dots + \beta_i = \alpha$. Hence, for some constant $C > 0$
$$
|D^\alpha(\varPhi\circ f)(x)|\leq C d(x, W)^{-|\beta_1|}\dots d(x, W)^{-|\beta_i|} = Cd(x, W)^{-|\alpha|}.
$$
\end{proof}

\medskip
\begin{defn}

Let $W$ be a closed semialgebraic subset of $\Real^n$ and let $Z\subset \Real^n\setminus W$. We will say that \emph{the property $\mathcal B_n(Z, W)$ holds}, if for any positive integer $p$ and for any $\eta > 0$ there exists a $\Lambda^0_p(W)$-regular function $\psi: \Real^n\setminus W\longrightarrow [0, 1]$ such that
\begin{equation}
\text{$\psi\equiv 1$ on $G_\rho(Z, W)$, with some $\rho\in (0, \eta)$, and}
\end{equation}
\begin{equation}
\text{$\supp$$\psi\subset G_\eta(Z, W).$}
\end{equation}

\end{defn}

\medskip
\begin{lem} If $W$ is a closed semialgebraic subset of $\Real^n$, $Z_1, Z_2, Z\subset \Real^n\setminus W$ and $\varepsilon, \eta\in (0, +\infty)$, then
$$
G_\eta(Z_1\cup Z_2, W) = G_\eta(Z_1, W)\cup G_\eta(Z_2, W) \quad \text{and} \quad G_\varepsilon(G_\eta(Z, W), W)\subset G_{\varepsilon + \eta + \varepsilon\eta}(Z, W).
$$
\end{lem}

\begin{proof} The first being straightforward, we will check the second inclusion. Let $x\in G_\varepsilon(G_\eta(Z, W), W)$. Hence, $d(x, G_\eta(Z, W)) < \varepsilon d(x, W)$. It follows that there exists $y\in G_\eta(Z, W)$ such that $|x - y| <  \varepsilon d(x, W)$. On the other hand $d(y, Z)< \eta d(y, W)$; thus,
$$
d(x, Z)\leq |x - y| + d(y, Z) < \varepsilon d(x, W) + \eta d(y, W)\leq
\varepsilon d(x, W) + \eta [|x - y| + d(x, W)] <
$$
$$
\varepsilon d(x, W) + \eta\varepsilon d(x, W) + \eta d(x, W) = (\varepsilon + \eta + \varepsilon\eta) d(x, W).
$$
\end{proof}

\medskip
\begin{lem}\label{suma}

If $W$ is a closed semialgebraic subset of $\Real^n$ and $Z_1,\dots, Z_k\subset \Real^n\setminus W$ and if $\mathcal B_n(Z_i, W)$ holds for every $i\in \{1,\dots, k\}$, then $\mathcal B_n(\bigcup_{i = 1}^k Z_i, W)$ holds.

\end{lem}

\begin{proof} Take a piecewise polynomial $\mathcal C^p$-function $P: \Real\longrightarrow [0, 1]$ such that $P(t) = 1$, when $t\leq 1/3$, and $P(t) = 0$, when $t\geq 2/3$. For any given $\eta > 0$, let $\psi_i: \Real^n\setminus W\longrightarrow [0, 1]$ \, $(i\in \{1,\dots, k\})$ \, be a $\Lambda^0_p(W)$-regular function such that
$\psi_i = 1$ on $G_\rho(Z_i, W)$, \linebreak where $\rho\in (0, \eta)$,  and  $\supp$$\psi_i\subset G_\eta(Z_i, W)$. Since, by Lemma 3.7, $G_\eta(\bigcup_{i=1}^k Z_i, W) = \bigcup_{i=1}^k G_\eta(Z_i, W)$, the function
$$
\psi: \Real^n\setminus W\ni x \longmapsto 1 - P\Big(\sum_{i=1}^k \psi_i(x)\Big)\in [0, 1]
$$
is a $\varLambda^0_p(W)$-regular function (by Lemma 3.5) corresponding to $\bigcup_{i=1}^k Z_i$.
\end{proof}

\medskip
\begin{prop}
$\phantom{...............................................................................................................................................}$

\medskip
For any closed semialgebraic subset $W$ of $\Real^n$ and any semialgebraic $Z\subset \Real^n\setminus W$, the property $\mathcal B_n(Z, W)$ holds.
\end{prop}

\begin{proof}
We argue by induction on $m = \dim Z$. If $m = 0$, in view of Lemma \ref{suma}, one can assume that $Z = \{z\}$ is a singleton. If $\eta > 0$, then there exists $\rho\in (0, \eta)$ and $0 < r < R$ such that
$$
\text{$G_\rho(\{z\}, W)\subset B(z, r)\subset B(z, R)\subset G_\eta(\{z\}, W)$},
$$
where $B(z, r) := \{x\in \Real^n: \, |x - z|\leq r\}$. Now, it is enough to take a semialgebraic $\mathcal C^p$-function $\psi:\Real^n\longrightarrow [0, 1]$ such that $\psi = 0$ on $\Real^n\setminus B(z, R)$ and $\psi = 1$ on $B(z, r)$.

\medskip
Let now $m\in\{1,\dots, n-1\}$ and assume that $\mathcal B_n(Z', W)$ holds for any semialgebraic subset $Z'\subset\Real^n\setminus W$, such that $\dim Z' < m$.

\medskip
By Theorem 2.6 applied to the sets $Z$ and $W$ and to the Lipschitz function $g(x) := d(x, W)$, combined with Lemma \ref{suma} and the induction hypothesis, we reduce the general case to that where $Z$ is a $\Lambda_p$-regular cell (2.3);
$$
Z = \{(u, w): \, u\in T, w = \varphi(u)\},
$$
where $u = (x_1,\dots, x_m)$, $w = (x_{m+1},\dots, x_n)$, \, $T$ is an open $\Lambda_p$-regular cell in $\Real^m$, and $\varphi: T\longrightarrow \Real^{n-m}$ is a semialgebraic $\Lambda_p$-regular mapping, and moreover, the function
$$
T\ni u \longmapsto d\big((u,\varphi(u)), W\big)\in \Real
$$
is $\Lambda_p$-regular.

\medskip
It is elementary that if $M\geq 0$ is a Lipschitz constant of the mapping $\varphi$, then putting $L := 1/\sqrt{1 + M^2}$, we have

\begin{equation}
\forall x = (u, w)\in T\times \Real^{n-m}: \, L|w - \varphi(u)| \leq d(x, Z) \leq |w - \varphi(u)|
\end{equation}

and

\begin{equation}
\forall \, x\in \Real^n\setminus (T\times \Real^{n-m}): \, d(x, Z) \geq L d(x, \partial Z).
\end{equation}

\medskip
Take any $\eta$ such that
\begin{equation}
0 < \eta < L.
\end{equation}

Fix any $\eta'\in (0, \eta)$. By the induction hypothesis applied to $Z' := \partial Z \setminus W$, where $\partial Z := \overline{Z}\setminus Z$, there exists a $\Lambda^0_p(W)$-regular function $\lambda: \Real^n\setminus W\longrightarrow [0, 1]$ \, such that $\supp$$\lambda\subset G_{\eta'}(Z', W)$ \, and \, $\lambda \equiv 1$ \, on $G_{\rho'}(Z', W)$, \, for some $\rho'\in (0, \eta')$.

\medskip
Put\
$$
\psi(x) = \psi(u, w) := \big(1 - \lambda(x)\big)P\Big(\frac{|w - \varphi(u)|^2}{\gamma d\big((u,\varphi(u)), W\big)^2}\Big) + \lambda(x),
$$
where $x = (u, w)\in T\times \Real^{n-m}$, $P$ is a function from the proof of Lemma \ref{suma} and $\gamma > 0$ is a constant to be carefully chosen. We will show that the function $\psi$ extends by means of $\lambda$ to a $\Lambda^0_p(W)$-regular function $\psi: \Real^n\setminus W\longrightarrow [0, 1]$, provided that $\gamma > 0$ is sufficiently small.

\medskip
Fix any $\delta \in (0, L)$.
According to (3.4), the set
$$
H := \{x\in \Real^n\setminus W: \, d(x, Z) > \delta d(x, \partial Z)\}\cup G_{\rho'}(Z', W)
$$
is an open neighborhood of the set $\big[\Real^n\setminus (T\times \Real^{n-m})\big]\setminus W$ in the set $\Real^n\setminus W$.

\begin{lem} We claim that if $\gamma > 0$ is sufficiently small, then $\psi = \lambda$ on $(T\times \Real^{n-m})\cap H$.
\end{lem}

\medskip
Indeed, let $x\in (T\times \Real^{n-m})\cap H$. If $x\in G_{\rho'}(Z', W)$, then clearly $\psi(x) = \lambda(x)$, so let us assume that $x\not\in G_{\rho'}(Z', W)$; i.e.

\begin{equation}
d(x, \partial Z\setminus W) \geq {\rho'} d(x, W).
\end{equation}

\medskip
The following two cases are possible: $d(x, \partial Z) = d(x, (\partial Z)\setminus W)$, or

 $d(x, \partial Z) = d(x, (\partial Z)\cap W)$.

\medskip
In the first case, we have in view of (3.6)
$$
d\big((u, \varphi(u)), W\big) \leq |(u, \varphi(u)) - x| + d(x, W)
$$
$$
\leq |w - \varphi(u)| + (1/\rho')d(x, \partial Z) \leq |w - \varphi(u)| + \big(1/(\rho'\delta)\big)d(x, Z)
$$
$$
\leq |w - \varphi(u)| + \big(1/(\rho'\delta)\big)|w - \varphi(u)|.
$$
Hence
$$
\frac{|w - \varphi(u)|^2}{\gamma d\big((u, \varphi(u)), W\big)^2} \geq \frac{(\rho'\delta)^2}{\gamma(1 + \rho'\delta)^2} > 2/3,
$$
if only

\begin{equation}
0 < \gamma < \frac{3(\rho'\delta)^2}{2(1 + \rho'\delta)^2}.
\end{equation}

\medskip
In the second case, we have
$$
d\big((u, \varphi(u)), W\big) \leq |(u, \varphi(u)) - x| + d(x, W)
$$
$$
\leq |w - \varphi(u)| + d(x, (\partial Z)\cap W)  = |w - \varphi(u)| + d(x, \partial Z)
$$
$$
< |w - \varphi(u)| + (1/\delta)d(x, Z) \leq |w - \varphi(u)| + (1/\delta)|w - \varphi(u)|.
$$

\medskip
Hence, if $\gamma$ satisfies (3.7), then we have again
$$
\frac{|w - \varphi(u)|^2}{\gamma d\big((u, \varphi(u)), W\big)^2} \geq \frac{\delta^2}{\gamma(1 + \delta)^2} > \frac{(\rho'\delta)^2}{\gamma( 1 + \rho'\delta)^2} > 2/3,
$$
since $\rho' < 1$.

\medskip
It follows that if $\gamma$ satisfies (3.7), then
$$
P\Big(\frac{|w - \varphi(u)|^2}{\gamma d\big((u,\varphi(u)), W\big)^2}\Big) = 0,
$$
hence $\psi(x) = \lambda(x)$, which ends the proof of Lemma 3.10.

\medskip
Now, we will show that, if $\gamma > 0$ satisfies (3.7), then $\supp$$\psi\subset G_\eta(Z, W)$. Let $x\in \Real^n\setminus W$ and $x\not\in G_{\eta'}(Z, W)$, so

\begin{equation}
d(x, Z) \geq \eta' d(x, W).
\end{equation}

\medskip
In the case when $x\in H$, we have $d(x, Z')\geq d(x, Z) \geq \eta' d(x, W)$; hence, $\psi(x) = \lambda(x) = 0$. In the case when $x\not\in H$, we have in particular that $x\in T\times \Real^{n-m}$. As before, $d(x, Z')\geq d(x, Z) \geq \eta' d(x, W)$; hence, $\lambda(x) = 0$. Moreover,

$$
d\big((u,\varphi(u)), W\big)\leq |(u, \varphi(u)) - x| + d(x, W)
$$
$$
\leq |w - \varphi(u)| + (1/\eta')d(x, Z) \leq |w - \varphi(u)|\frac{\eta' + 1}{\eta'};
$$
hence, if $\gamma$ satisfies (3.7),
$$
\frac{|w - \varphi(u)|^2}{\gamma d\big((u, \varphi(u)), W\big)^2} \geq \frac{(\eta')^2}{\gamma(1 + \eta')^2} > \frac{(\rho'\delta)^2}{\gamma (1 + \rho'\delta)^2)} > 2/3,
$$
consequently
$$
P\Big(\frac{|w - \varphi(u)|^2}{\gamma d\big((u,\varphi(u)), W\big)^2}\Big) = 0,\quad\text{thus \, $\psi(x) = \lambda(x) = 0$}.
$$
It follows that $\supp$$\psi\subset G_\eta(Z, W)$.

\medskip
Now we will find $\rho\in (0, \eta)$ such that $\psi\equiv 1$ on $G_\rho(Z, W)$. Assume first that
\begin{equation}
0 < \rho < \rho'\delta
\end{equation}
and take any $x\in G_\rho(Z, W)$; i.e. $d(x, Z) < \rho d(x, W)$.

\medskip
If $x\in G_{\rho'}(Z', W)$, then $\psi(x) = \lambda(x) = 1$, so in what follows we can assume that $x\not\in G_{\rho'}(Z', W)$; i.e.
\begin{equation}
d(x, Z')\geq \rho' d(x, W).
\end{equation}

Consider two possible cases exactly as in the proof of Lemma 3.10. If $d(x, \partial Z) = d(x, Z')$, then by (3.9) and (3.10)
$$
d(x, Z) < \rho d(x, W) < \rho'\delta d(x, W) \leq \delta d(x, Z') = \delta d(x, \partial Z),
$$
which implies that $x\not\in H$. If $d(x, \partial Z) = d(x, (\partial Z)\cap W)$, then
$$
d(x, Z) < \rho d(x, W) < \rho'\delta d(x, W) < \delta d(x, W) \leq \delta d(x, (\partial Z)\cap W) = \delta d(x, \partial Z),
$$
which again implies that $x\not\in H$.

\medskip
Consider now $x\in G_\rho(Z, W)\setminus H$. Then in particular $x = (u, w)\in T\times \Real^{n-m}$ and, according to (3.3),
$$
L|w -\varphi(u)| \leq d(x, Z) < \rho d(x, W) \leq \rho |x - (u, \varphi(u))| + \rho d\big((u, \varphi(u)), W\big)
$$
$$
= \rho |w - \varphi(u)| + \rho d\big((u, \varphi(u)), W\big).
$$

\medskip
Hence, if we assume that

\begin{equation}
 0 < \rho < L\frac{\sqrt{\gamma}}{\sqrt{3} + \sqrt{\gamma}},
\end{equation}

then

$$
\frac{|w - \varphi(u)|^2}{\gamma d\big((u, \varphi(u)), W\big)^2}\leq \frac{\rho^2}{\gamma(L - \rho)^2} < 1/3;
$$

consequently,

$$
P\Big(\frac{|w - \varphi(u)|^2}{\gamma d\big((u,\varphi(u)), W\big)^2}\Big) = 1, \quad\text{thus \,\, $\psi(x) = 1$}.
$$

We conclude that $\psi\equiv 1$ on  $G_\rho(Z, W)$, if only $\rho$ satisfies (3.9) and (3.11).

\medskip
Now we will check that $\psi: \Real^n\longrightarrow [0, 1]$ is $\Lambda^0_p(W)$ regular. Since $\psi = \lambda$ on $H$ and $\lambda$ is $\Lambda^0_p(W)$-regular, due to induction hypothesis, it suffices to check $\Lambda^0_p(W)$-regularity on $\Real^n\setminus (\overline{H}\cup W)$. Moreover, since $\supp$$\psi\subset G_\eta(Z, W)$ and $\psi\equiv 1$ on $G_\rho(Z, W)$, it suffices to check $\Lambda^0_p(W)$-regularity, assuming that
\begin{equation}
x\in \Real^n\setminus (\overline{H}\cup W),\,  d(x, Z) < \eta d(x, W),
\end{equation}
$\phantom{.......................................} \text{and} \,\, d(x, Z') > \rho'd(x, W).$

\medskip
For $x = (u, w)\in T\times \Real^{n-m}$ satisfying (3.12), we have by (3.3) and (3.5) that
$$
d(x, W) \leq d\big((u, \varphi(u)), W\big) + |x - (u, \varphi(u))|
$$
$$
\leq d\big((u, \varphi(u)), W\big) + (1/L)d(x, Z) < d\big((u, \varphi(u)), W\big) + (\eta/L)d(x, W);
$$
consequently,
\medskip
\begin{equation}
d(x, W) < \frac{L}{L - \eta}d\big((u, \varphi(u)), W).
\end{equation}

Since by (3.3), (3.12) and (3.13)
$$
\frac{|w - \varphi(u)|}{d\big((u, \varphi(u)), W\big)} \leq \frac{(1/L) d(x, Z)}{\big(1 - (\eta/L)\big)d(x, W)} < \frac{\eta}{L - \eta},
$$
and
$$
\frac{|w - \varphi(u)|^2}{d\big((u, \varphi(u)), W\big)^2} = \sum_{j = m+1}^n\Big[\frac{x_j - \varphi_j(u)}{d\big((u, \varphi(u)), W\big)}\Big]^2,
$$
where $\varphi = (\varphi_{m+1},\dots, \varphi_n)$, it follows from Lemmas 3.4, 3.3 and 3.5 consecutively applied, that it suffices to check that every function $f_j(x) = f_j(u, w) := \varphi_j(u)$ and the function $g(x) = g(u, w) :=d\big((u, \varphi(u)), W\big)$ are $\Lambda^1_p(W)$-regular\footnote{$x_j$ was omitted as obviously $\Lambda^1_p(W)$-regular} on the set (3.12).

\medskip
To this end, take any $\alpha\in \Nat^n\setminus\{0\}$ such that $|\alpha|\leq p$. Then for any $x = (u, w)$ satisfying (3.12),

$$
|D^\alpha f_j(x)|\leq C d(u, \partial T)^{1 - |\alpha|} \leq C L^{1-|\alpha|}d\big((u, \varphi(u)),\partial Z)^{1-|\alpha|}
$$
$$
\leq C L^{1-|\alpha|}\max\Big[d\big((u,\varphi(u)), Z'\big)^{1-|\alpha|}, d\big((u, \varphi(u)), (\partial Z)\cap W\big)^{1-|\alpha|}\Big],
$$
where $C$ is a positive constant.

\smallskip
On the other hand, by (3.3) and (3.12),
$$
d\big((u, \varphi(u)), Z'\big)\geq d(x, Z') - |w - \varphi(u)|\geq d(x, Z') - (1/L)d(x, Z)
$$
$$
\geq d(x, Z') - (\delta/L)d(x, \partial Z) \geq d(x, Z') - (\delta/L)d(x, Z') \geq \rho'\Big(1 - \frac{\delta}{L}\Big)d(x, W),
$$
and, by (3.13),
$$
d\big((u, \varphi(u)), (\partial Z)\cap W\big)\geq d\big((u, \varphi(u)), W\big) > \Big(1 - \frac{\eta}{L}\Big)d(x, W).
$$
It follows that $|D^\alpha f_j(x)|\leq\tilde{C}d(x, W)^{1- |\alpha|}$, where $\tilde{C}$ is a positive constant. The same estimate holds for $g$, which ends the proof that $\psi$ is $\Lambda^0_p(W)$-regular.

\medskip
To finish the proof of Proposition 3.9, it remains to consider the case $m = n$; i.e. $Z$ is an open semialgebraic subset of $\Real^n\setminus W$.
Let $\eta > 0$. By induction hypothesis applied to $Z':= \partial Z\setminus W$, there exists a $\Lambda^0_p(W)$-regular function $\lambda: \Real^n\setminus W\longrightarrow [0, 1]$ such that $\supp$$\lambda\subset G_\eta(Z', W)$ and $\lambda\equiv 1$ on $G_\rho(Z', W)$, for some $\rho\in (0, \eta)$. Now, we define
$$
\psi(x) := \begin{cases} 1, &\text{when $x\in Z$}\\
                  \lambda(x), &\text{when $x\in \big[(\Real^n\setminus W)\setminus \overline{Z}\big]\cup G_\rho(Z', W)$.} \end{cases}
$$

\medskip
Clearly, $\psi: \Real^n\setminus W\longrightarrow [0, 1]$ is $\Lambda^0_p(W)$-regular, \, $\supp$$\psi\subset G_\eta(Z, W)$ \, and \, $\psi\equiv 1$ on $G_\rho(Z, W)$.
\end{proof}

\medskip
\begin{proof}[Proof of Theorem 3.2]

By Proposition 3.9, for each $i\in\{1,\dots, s\}$, there exists a $\Lambda^0_p(W)$-regular function $\psi_i: \Real^n\setminus W\longrightarrow [0, 1]$ such that $\supp$$\psi_i\subset G_\eta(U_i, Z)$ and $\psi_i \equiv 1$ on $G_{\rho_i}(U_i, W)$, for some $\rho_i\in (0, \eta)$. By Lemmas 3.4 and 3.3, the functions
$$
\omega_i := \frac{\psi_i}{\psi_1 + \dots + \psi_s} \quad (i\in \{1,\dots, s\}).
$$
are the required partition of unity.
\end{proof}

\bigskip
\section{Proof of Theorem 1.3}

\medskip
 By Theorem 2.6, applied to $g$ and the set $W$, we obtain a $\Lambda_p$-regular stratification \linebreak $\Real^n\setminus W = C_1\cup \dots\cup C_s$ of the set $\Real^n\setminus W$, such that, for each $i\in \{1,\dots, s\}$, the stratum $C_i$, after an orthogonal linear change of coordinates in $\Real^n$, is a $\Lambda_p$-regular cell in $\Real^n$ and if $C_i$ is open then $g|C_i$ is $\Lambda_p$-regular, while in the case $\dim C_i = m < n$, when $C_i$ is of the form
\begin{equation}
C_i = \{(u, w)\in D_i\times \Real^{n-m}:w = \varphi_i(u)\},
\end{equation}
$\phantom{...........}$where $D_i$ is open in $\Real^m$ and $\varphi_i: D_i\longrightarrow \Real^{n-m}$ is $\Lambda_p$-regular,
then
\begin{equation}
\text{the mapping $D_i\ni u\longmapsto g(u, \varphi_i(u))\in \Real^d$ is $\Lambda_p$-regular.}
\end{equation}
Additionally, without any loss in generality, we can assume that
\begin{equation}
\dim C_1 \leq \dim C_2 \leq \dots \leq \dim C_s.
\end{equation}
If $C_i$ is of the form (4.1) and $M_i$ is a Lipschitz constant of $\varphi_i$, then we put \linebreak $L_i := 1/\sqrt{1 + M_i^2}$. If $\dim C_i = n$, we put $L_i := 1$. Let $A$ be a Lipschitz constant of $g$.

To simplify the notation, we will write in this section $G_\eta(Z)$ in the place of $G_\eta(Z, W)$, for any $\eta > 0$ and any semialgebraic subset $Z\subset \Real^n\setminus W$. This will not lead to a confusion because the set $W$ is fixed in this section.

\medskip
Given any $\kappa >0$ as in Theorem 1.3, fix any  $\theta\in (0, 1)$ so small that $A(\theta/L_i) <\kappa$, for each $i\in\{1,\dots, s\}$. We define by induction on $i\in\{1,\dots, s\}$, a sequence of semialgebraic sets $Z_1\subset C_1, \dots, Z_s\subset C_s$, and two sequences of positive numbers \newline $\eta_s < \delta_s < \eta_{s-1} < \delta_{s-1} < \dots < \eta_1 < \delta_1 < \theta$ such that

\medskip
\begin{equation}
d(x, C_1\cup\dots\cup C_i) < \eta_id(x, W) \Longrightarrow \phantom{..........................................}
\end{equation}
$\phantom{.............................}x\in G_{\eta_i}(Z_i)\cup G_{\eta_i}(G_{\eta_{i-1}}(Z_{i-1}))\cup\dots\cup G_{\eta_i}(G_{\eta_{i-1}}(\dots(G_{\eta_1}(Z_1))\dots ));$

\begin{equation}
\text{there exists a $\Lambda^1_p(W)$-regular function $f_i: G_{\delta_i}(Z_i)\longrightarrow \Real\phantom{.........}$}
\end{equation}
$\phantom{.....................}\text{such that} \quad \forall \, x\in G_{\delta_i}(Z_i): \, |f_i(x) - g(x)|\leq A(\delta_i/L_i) d(x, W)$;

\begin{equation}
\text{for every  $j\in\{1,\dots, i\}$ there exists $\varepsilon_{ij} \in (0, \delta_j)$ such that $\phantom{........}$}
\end{equation}
$\phantom{.....................} G_{\eta_i}(G_{\eta_{i-1}}(\dots(G_{\eta_j}(Z_j))\dots ))\subset G_{\varepsilon_{ij}}(Z_j).$

\medskip
To begin the inductive definition, we put $Z_1 := C_1$. Since $C_1$ is the first stratum, its boundary $\partial Z_1 := \overline{Z_1}\setminus Z_1$ is contained in $W$. Take any $\delta_1 < \min\{\theta, L_1\}$. Then, for each $x\in G_{\delta_1}(Z_1)$, we have
$$
d(x, Z_1)< \delta_1 d(x, W) \leq \delta_1 d(x, \partial Z_1);
$$
hence, by (3.4),  $x = (u, w)\in D_1\times \Real^{n- m_1}$, \, where \, $m_1 = \dim C_1$. Therefore, we can define
$$
f_1(x) = f_1(u, w) := g(u, \varphi_1(u)).
$$
Then, by (3.3),
$$
|f_1(x) - g(x)| = |g(u, \varphi_1(u)) - g(u, w)|\leq A|w - \varphi_1(u)|
$$
$$
\leq AL_1^{-1}d(x, C_1)\leq A(\delta_1/L_1)d(x, W).
$$
To check that $f_1$ is $\Lambda^1_p(W)$-regular, we take any $\alpha\in \Nat^n\setminus\{0\}$ such that $|\alpha|\leq p$. Then we have by (4.2)
$$
|D^\alpha f_1(x)|\leq B_1d(u, \partial D_1)^{1-|\alpha|}\leq B_1(L_1)^{1-|\alpha|}d\big((u, \varphi_1(u)), \partial Z_1\big)^{1-|\alpha|}
$$
$$
\leq B_1(L_1)^{1-|\alpha|}\Big[d(x, \partial Z_1) - |w - \varphi_1(u)|\Big]^{1-|\alpha|}
$$
$$
\leq B_1(L_1)^{1-|\alpha|}\Big[d(x, \partial Z_1) - L_1^{-1}d(x, Z_1)\Big]^{1-|\alpha|}
$$
$$
\leq B_1(L_1)^{1-|\alpha|}\Big(1 - \frac{\delta_1}{L_1}\Big)^{1-|\alpha|}d(x, W)^{1 - |\alpha|},
$$
where $B_1$ is a positive constant. Fix any $\eta_1\in (0, \delta_1)$ and put $\varepsilon_{11} :=\eta_1$.

\medskip
To define $Z_{i+1}$, where $i < s$, observe that, due to (4.3) and (4.4),
$$
(\partial C_{i+1})\setminus W \subset C_1\cup\dots\cup C_i\subset G_{\eta_i}(Z_i)\cup\dots\cup G_{\eta_i}(G_{\eta_{i-1}}(\dots(G_{\eta_1}(Z_1))\dots))
$$.
Put
$$
Z_{i+1} := C_{i+1}\setminus\Big[G_{\eta_i}(Z_i)\cup\dots\cup G_{\eta_i}(G_{\eta_{i-1}}(\dots(G_{\eta_1}(Z_1))\dots))\Big].
$$

\medskip
By (4.4)
\begin{equation}
\forall \, z\in Z_{i+1}: \,\, d(z, \partial C_{i+1}) = \min\big\{d(z, (\partial C_{i+1})\setminus W), d(z, (\partial C_{i+1})\cap W)\big\}\geq
\end{equation}
$$
\phantom{....}\min\big\{d(z, C_1\cup\dots\cup C_i), d(z, W)\big\} \geq \eta_id(z, W).
$$

\medskip
Assume first that $\dim C_{i+1} < n$. Then we choose $\delta_{i+1}\in (0, \eta_i/(1 + \eta_i))$ in such a way that
\begin{equation}
\frac{\delta_{i+1}}{\eta_i - \delta_{i+1}(\eta_i + 1)} < L_{i+1}.
\end{equation}

\medskip
We will now check that
\begin{equation}
G_{\delta_{i+1}}(Z_{i+1}) \subset D_{i+1}\times \Real^{n- m_{i+1}}.
\end{equation}
Indeed, take any $x\in G_{\delta_{i+1}}(Z_{i+1})$. There exists $z\in Z_{i+1}$ such that $|x - z|< \delta_{i+1}d(x, W)$.
By (4.7), we have
$$
|x - z| < \delta_{i+1}\big(|x - z| + d(z, W)\big)\leq
\leq \delta_{i+1}|x - z| + \frac{\delta_{i+1}}{\eta_i}d(z, \partial C_{i+1})\leq
$$
$$
\big(\delta_{i+1} + \frac{\delta_{i+1}}{\eta_i}\big)|x - z| + \frac{\delta_{i+1}}{\eta_i}d(x, \partial C_{i+1})
$$
hence,
$$
d(x, C_{i+1}) \leq |x - z| < \frac{\delta_{i+1}}{\eta_i - \delta_{i+1}(\eta_i + 1)}d(x, \partial C_{i+1})
$$
which, in view of (4.8) and (3.4), implies that $x\in D_{i+1}\times \Real^{n - m_{i+1}}$.

\medskip
In view of (4.9), the following definition of $f_{i+1}: G_{\delta_{i+1}}(Z_{i+1})\longrightarrow \Real$ is possible
$$
f_{i+1}(x) = f_{i+1}(u, w) := g(u,\varphi_{i+1}(u)).
$$
Then, we have
$$
|f_{i+1}(x) - g(x)|\leq A|w - \varphi_{i+1}(u)|\leq (A/L_{i+1})d(x, C_{i+1})
$$
$$
\leq (A/L_{i+1})d(x, Z_{i+1})< A(\delta_{i+1}/L_{i+1})d(x, W).
$$

\smallskip
Now we want to check that $f_{i+1}$ is $\Lambda^1_p(W)$-regular. Let $\alpha\in \Nat^n\setminus\{0\}$ such that $|\alpha|\leq p$ and let $x\in G_{\delta_{i+1}}(Z_{i+1})$. Then there exists $z\in Z_{i+1}$ such that $|x - z| < \delta_{i+1}d(x, W)$. By (4.2) and (4.7),  we get
$$
|D^\alpha f_{i+1}(x)|\leq B_{i+1}d(u, \partial D_{i+1})^{1 - |\alpha|}
$$
$$
\leq B_{i+1}(L_{i+1})^{1-|\alpha|}d\big((u, \varphi_{i+1}(u)), \partial C_{i+1}\big)^{1-|\alpha|}
$$
$$
\leq B_{i+1}(L_{i+1})^{1 - |\alpha|}\Big[d(z, \partial C_{i+1}) - |(u, \varphi_{i+1}(u)) - z|\Big]^{1 - |\alpha|}
$$
$$
\leq B_{i+1}(L_{i+1})^{1 - |\alpha|}\Big[\eta_id(z, W) - |(u, \varphi_{i+1}(u)) - z|\Big]^{1 - |\alpha|}
$$
$$
\leq B_{i+1}(L_{i+1})^{1 - |\alpha|}\Big[\eta_i d(x, W) - \eta_i|x - z| - |x - z| - |(u, \varphi_{i+1}(u)) - x|\Big]^{1 - |\alpha|}
$$
$$
\leq B_{i+1}(L_{i+1})^{1 - |\alpha|}\Big[\eta_i d(x, W) - (\eta_i + 1)|x - z| - |w - \varphi_{i+1}(u)|\Big]^{1 - |\alpha|}
$$
$$
\leq B_{i+1}(L_{i+1})^{1 - |\alpha|}\Big[\eta_i d(x, W) - (\eta_i + 1)\delta_{i+1}d(x, W) - (1/L_{i+1})d(x, C_{i+1})\Big]^{1 - |\alpha|}
$$
$$
\leq B_{i+1}(L_{i+1})^{1 - |\alpha|}\Big[\eta_i d(x, W) - (\eta_i + 1)\delta_{i+1}d(x, W) - (1/L_{i+1})d(x, Z_{i+1})\Big]^{1 - |\alpha|}
$$
$$
\leq B_{i+1}(L_{i+1})^{1 - |\alpha|}\Big[\eta_i d(x, W) - (\eta_i + 1)\delta_{i+1}d(x, W) - (\delta_{i+1}/L_{i+1})d(x, W)\Big]^{1 - |\alpha|}
$$
$$
= B_{i+1}[\eta_i - \delta_{i+1}(\eta_i +1)]^{1- |\alpha|}\Big(L_{i+1} - \frac{\delta_{i+1}}{\eta_i - \delta_{i+1}(\eta_i + 1)}\Big)^{1-|\alpha|}d(x, W)^{1-|\alpha|},
$$
where $B_{i+1}$ is a positive constant.

\medskip
Assume now that $\dim C_{i+1} = n$. Then we fix any $\delta_{i+1} \in (0, \eta_i/(1+\eta_i))$. \linebreak If $x\in G_{\delta_{i+1}}(Z_{i+1})$, then there exists $z\in Z_{i+1}$ such that $|x - z|< \delta_{i+1}d(x, W)$, and by (4.7)
$$
|x - z| < \delta_{i+1}|x - z| + \delta_{i+1}d(z, W)\leq \delta_{i+1}|x - z| + \frac{\delta_{i+1}}{\eta_i}d(z, \partial C_{i+1});
$$
thus,
$$
|x - z| < \frac{\delta_{i+1}}{(1 - \delta_{i+1})\eta_i}d(z, \partial C_{i+1}) < d(z, \partial C_{i+1}).
$$
It follows that
\begin{equation}
G_{\delta_{i+1}}(Z_{i+1})\subset C_{i+1};
\end{equation}
hence we can define $f_{i+1}$ as the restriction $g|C_{i+1}$ of $g$ to $C_{i+1}$. Clearly, $f_{i+1}$ is $\varLambda_p^1(W)$-regular and corresponding condition (4.5) is trivially satisfied, since then the left-hand side is identically zero.

\medskip
Now we will need to specify $\eta_{i+1}\in (0, \delta_{i+1})$. To this end we will need the properties of the operation $G_\delta$ expressed in Lemma 3.7.  We take any $\eta_{i+1}\in (0, \delta_{i+1})$ so small that
$$
\varepsilon_{i+1,j}:= \varepsilon_{ij} + \eta_{i+1} + \eta_{i+1}\varepsilon_{ij} < \delta_j \qquad (j\in\{1,\dots,i\})
$$
and $\varepsilon_{i+1,i+1}:=\eta_{i+1}$, both in the case $\dim C_{i+1} < n$ as well as when $\dim C_{i+1} = n$. This choice ensures (4.6) for $i+1$ in the place of $i$, according to Lemma 3.7.

\medskip
Now we will check the property (4.4) for $i+1$ in the place of $i$.

\medskip
Let  $d(x, C_1\cup\dots\cup C_i\cup C_{i+1}) < \eta_{i+1}d(x, W)$.

\medskip
If $d(x, C_1\cup\dots\cup C_i\cup C_{i+1}) = d(x, C_1\cup\dots\cup C_i)$, then by (4.4), \linebreak $x\in G_{\eta_i}(Z_i)\cup G_{\eta_i}(G_{\eta_{i-1}}(Z_{i-1}))\cup\dots\cup G_{\eta_i}(G_{\eta_{i-1}}(\dots(G_{\eta_1}(Z_1))\dots ))$.

\medskip
If $d(x, C_1\cup\dots\cup C_i\cup C_{i+1}) = d(x, Z_{i+1})$, then certainly $d(x, Z_{i+1}) < \eta_{i+1}d(x, W)$; thus, $x\in G_{\eta_{i+1}}(Z_{i+1})$.

\medskip
It remains the case, when $d(x, C_1\cup\dots\cup C_i\cup C_{i+1}) = $
$$
 d\Big(x, C_{i+1}\cap \big[G_{\eta_i}(Z_i)\cup G_{\eta_i}(G_{\eta_{i-1}}(Z_{i-1}))\cup\dots\cup G_{\eta_i}(G_{\eta_{i-1}}(\dots(G_{\eta_1}(Z_1))\dots
 ))\big]\Big).
$$

Then
$$
d\Big(x, G_{\eta_i}(Z_i)\cup G_{\eta_i}(G_{\eta_{i-1}}(Z_{i-1}))\cup\dots\cup G_{\eta_i}(G_{\eta_{i-1}}(\dots(G_{\eta_1}(Z_1))\dots
 ))\Big)
$$
$\phantom{.........................................................................................}< \eta_{i+1}d(x, W)$;
thus,
$$
x\in G_{\eta_{i+1}}(G_{\eta_i}(Z_i))\cup G_{\eta_{i+1}}(G_{\eta_i}(G_{\eta_{i-1}}(Z_{i-1})))\cup\dots\phantom{.........................}
$$
$$
\phantom{..........................................}\dots\cup G_{\eta_{i+1}}(G_{\eta_i}(G_{\eta_{i-1}}(\dots(G_{\eta_1}(Z_1))\dots
 ))).
$$

\bigskip
To finish the proof of Theorem 1.3, we put
$$
U_i:= G_{\eta_s}(\dots (G_{\eta_i}(Z_i))\dots),\quad \text{for $i\in \{1,\dots,s\}$},
$$
and choose $\eta > 0$ so small that $G_\eta(U_i)\subset G_{\delta_i}(Z_i)$, for each $i\in \{1,\dots, s\}$ (see (4.6) and Lemma 3.7).
It follows from (4.4) that $U_1,\dots, U_s$ is a covering of $\Real^n\setminus W$. We take now the partition of unity $\{\omega_i\}$ \,\, $(i\in \{1,\dots, s\})$ adapted to this covering and to $\eta$ according to Theorem 3.2. In virtue of Lemma 3.3, for each $i\in\{1,\dots, s\}$, the function $f_i\omega_i$ is $\varLambda_p^1(W)$-regular on $G_{\delta_i}(Z_i)$ and obviously extends by zero to a $\varLambda_p^1(W)$-regular function defined on $\Real^n\setminus W$ and, by (4.5), for each $x\in \Real^n\setminus W$,
$$
|f_i(x)\omega_i(x) - g(x)\omega_i(x)| \leq A(\delta_i/L_i)d(x, W)\omega_i(x) \leq A(\theta/L_i)d(x, W)\omega_i(x)< \kappa d(x, W)\omega_i(x) .
$$
Hence the function
$$
f := f_1\omega_1 + \dots + f_s\omega_s,
$$
is $\varLambda_p^1(W)$-regular on $\Real^n\setminus W$ and for each $x\in \Real^n\setminus W$
$$
|f(x) - g(x)|\leq \sum_{i=1}^s|f_i(x)\omega_i(x) - g(x)\omega_i(x)|\leq \sum_{i=1}^s\kappa d(x, W)\omega_i(x) = \kappa d(x, W),
$$
which ends the proof of Theorem 1.3.

\bigskip
\section{Two applications}

\bigskip
We give here two almost immediate consequences of Theorem 1.1. The first one is another proof of a theorem of Bierstone, Milman and Paw{\l}ucki (cf. [3, C.11]) that, given any positive integer $p$, any closed semialgebraic (or, more generally, definable in some o-minimal structure $\mathcal S$) subset $W$ of $\Real^n$ is the zero-set of some semialgebraic (respectively, definable in $\mathcal S$) $\mathcal C^p$-function defined on $\Real^n$. We will prove the following.

\begin{thm}
$\phantom{.....................................................................................................................................}$
Let $W$ be a closed semialgebraic subset of $\Real^n$ and let $p$ be a positive integer. Then there exists a semialgebraic function $h: \Real^n \longrightarrow [0, \infty)$ of class $\mathcal C^p$, which is Nash on $\Real^n\setminus W$ and such that $W = h^{-1}(0)$. Moreover, $h$ is equivalent to the $(p+1)$-th power of the distance function from the set $W$.
\end{thm}

\medskip
In the proof we will use the following elementary Hestenes Lemma (cf. [15, Lemme 4.3]).

\begin{lem}
$\phantom{...............................................................................................................................}$
Let $W$ be a closed subset of an open subset $\varOmega$ of $\Real ^n$. If $h: \varOmega\setminus W\longrightarrow \Real$ is a $\mathcal C^p$-function and $$\lim_{x \to a}D^\alpha h(x) = 0,$$ for each $a\in W\cap\overline{\varOmega\setminus W}$ and each $\alpha\in \Nat^n$ such that $|\alpha|\leq p$, then $h$ extends by zero to a $\mathcal C^p$-function on $\varOmega$, $p$-flat on $W$.
\end{lem}

\begin{proof}[Proof of Theorem 5.1]

\medskip
By Theorem 1.1 there exists a semialgebraic $\Lambda^1_p(W)$-regular Nash function $f:\Real^n\setminus W \longrightarrow (0, \infty)$ equivalent to the function $\Real^n\setminus W\ni x \longmapsto d(x, W)$. When we put $h:=f^{p+1}$, we have, for any $\alpha\in \Nat^n$ such that $|\alpha|\leq p$ and any $x\in \varOmega\setminus W$
\begin{equation}
D^\alpha h(x) = \sum_{\beta_1 +\dots + \beta_{p+1} = \alpha}\frac{\alpha!}{\beta_1!\dots\beta_{p+1}!}D^{\beta_1}f(x)\dots D^{\beta_{p+1}}f(x).
\end{equation}
It follows that
$$
|D^\alpha h(x)|\leq \sum_{\beta_1 +\dots + \beta_{p+1} = \alpha}C\big(d(x, W)\big)^{1-|\beta_1|}\dots \big(d(x, W)\big)^{1-|\beta_{p+1}|}
$$
$$
\leq \tilde{C}\big(d(x, W)\big)^{p+1-|\alpha|}, \quad\text{where $C$ and $\tilde{C}$ are positive constants,}
$$
which implies that $\lim_{x\to a}D^\alpha h(x) = 0$, for each $a\in W\cap\overline{\Real^n\setminus W}$.
\end{proof}

\bigskip
Our second application concerns approximation of semialgebraic subsets by Nash compact hypersurfaces in the Hausdorff metric. To formulate the result let us denote by $\mathcal K_n$ the set of all nonempty compact subsets of $\Real^n$. Recall that the \emph{Hausdorff metric} on $\mathcal K_n$ is defined by the formula
$$
d_{\mathcal H}(A, B) :=\max\{\sup_{x\in A}d(x, B), \sup_{y\in B}d(y, A)\}.
$$

\medskip
\begin{thm}
$\phantom{...........................................................................................................................................}$
Let $W$ be a non-empty, compact, nowhere dense semialgebraic subset of $\Real^n$. Then there exists a semialgebraic family of Nash compact hypersurfaces $\{H_t\}$ \, $(0 < t < \theta)$ such that
$$
\lim_{t\to 0}d_{\mathcal H}(H_t, W) = 0.
$$
\end{thm}

\begin{proof}

\medskip
Let $f:\Real^n\setminus W\longrightarrow (0, \infty)$ be a Nash function such that for some positive constant $A$
\begin{equation}
\forall \, x\in \Real^n\setminus W: \,\, A^{-1}d(x, W) \leq f(x) \leq Ad(x, W).
\end{equation}
By Sard's theorem, the function $f$ has only finitely many critical values; hence, the set $H_t := f^{-1}(t)$ is a Nash hypersurface in $\Real^n$, for all $t > 0$ sufficiently small. We will show that $H_t$ are compact and their Hausdorff limit is $W$, when $t$ tends to $0$. For any subset $X$ of $\Real^n$ and any $\eta > 0$ put
$$
X^\eta := \{x\in \Real^n: \, d(x, X) < \eta \}.
$$
To this end, it is enough to show that, for any positive $\eta$ there exists $\delta > 0$ such that for each $t\in (0, \delta)$
\begin{equation}
H_t \subset W^\eta \qquad\text{and}\qquad W\subset H_t^\eta.
\end{equation}

\medskip
As for the first inclusion (5.3), let $x\in H_t$. Then by (5.2), $A^{-1}d(x, W) \leq t$; hence, if $t < A^{-1}\eta$, then $x\in W^\eta$.

\medskip
As for the second inclusion (5.3), let us define function
\begin{equation}
\lambda(\varepsilon) := \sup_{a\in W}d(a, \partial W^\varepsilon), \quad\text{for any $\varepsilon > 0$.}
\end{equation}

\medskip
We claim that
\begin{equation}
\lim_{\varepsilon\to 0} \lambda(\varepsilon) = 0.
\end{equation}
Otherwise, by the Curve Selection Lemma, there should exist a semialgebraic continuous arc $\gamma: (0, \xi)\longrightarrow W$, where $\xi > 0$, such that
$$
\lim_{\varepsilon \to 0}\gamma(\varepsilon) = a, \,\text{for some $a\in W$, and}
$$
$$
d(\gamma(\varepsilon), \partial W^\varepsilon) > \mu, \, \text{for some $\mu > 0$ and each $\varepsilon\in (0, \xi)$.}
$$
The last would mean that
$$
B(\gamma(\varepsilon), \mu)\subset \{x\in \Real^n: \, d(x, W)\leq \varepsilon\}, \, \text{for each $\varepsilon\in (0, \xi)$.}
$$
But $B(\gamma(\varepsilon), \mu)\rightarrow B(a, \mu)$ and $\{x\in \Real^n: \, d(x, W)\leq \varepsilon\}\rightarrow W$, as $\varepsilon \rightarrow 0$; hence, $B(a, \mu)\subset W$, a contradiction with our assumption that $W$ is nowhere dense.

\bigskip
Let $\eta > 0$. By (5.5), there exists $\delta > 0$ such that $\lambda(\delta) < \eta$. It follows that then
$$
d(a, \partial W^\delta) < \eta, \quad\text{for each $a\in W$}.
$$
Fix any $a\in W$. There exists $z\in \partial W^\delta$ such that $|a - z| = d(a, \partial W^\delta)$. Observe that the line segment $[a, z]$ has its endpoints respectively in $W$ and in $\partial W^\delta$. By (5.2),
$$
f^{-1}[0, \delta/A)\subset W^\delta;
$$
hence, there exists $y\in f^{-1}(\delta/A)\cap [a, z]$. Then  $|a - y|\leq |a - z| < \eta$ and $y\in H^\eta_{\delta/A}$, which shows that $W\subset H^\eta_{\delta/A}$.
\end{proof}

\medskip
Theorem 5.3 is deepened and generalized in our separate article [6].

\bigskip
\section{Final remarks}

\bigskip
Most of the results of our article remain true in a more general context of o-minimal structures expanding the field of real numbers $\Real$ with the same proofs, where the term \emph{semialgebraic} should be replaced by \emph{definable} and \emph{a Nash mapping} - by \emph{a definable $\mathcal C^\infty$ mapping}.

\medskip
As for Theorem 1.1, however, it relies heavily on the Efroymson-Shiota approximation theorem. Theorem 1.4 was generalized by A. Fischer [5, Theorem 1.1]) to o-minimal structures admitting $\mathcal C^\infty$-cell decompositions in which the exponential function is definable (the case of non-polynomially bounded structures).
For the case of general polynomially bounded o-minimal structures admitting $\mathcal C^\infty$-cell decompositions a big progress has been made recently by the last author and Guillaume Valette, who proved
that in such structures the Efroymson-Shiota approximation theorem holds true for $p\leq 1$ ([16, Theorem 4.8]). In particular, this result allows us to prove the following generalization of Theorem 5.1.

\begin{thm}

Let $\mathcal D$ be a polynomially bounded o-minimal structure expanding $\Real$ which admits $\mathcal C^\infty$-cell decompositions. Then, for any closed $\mathcal D$-definable subset $W$ of $\Real^n$ and any positive integer $p$, there exists a $\mathcal D$-definable $\mathcal C^p$-function $h:\Real^n\longrightarrow [0, \infty)$ which is $\mathcal C^\infty$ on $\Real^n\setminus W$ and such that $W = h^{-1}(0)$.

\end{thm}

\begin{proof}
By [16, Theorem 1.1], there exists a $\mathcal D$-definable $\mathcal C^\infty$-function $f: \Real^n\setminus W\longrightarrow [0, \infty)$ such that
$$
\forall \, x\in \Real^n\setminus W: A^{-1}f(x) \leq d(x, W) \leq Af(x),
$$
where $A$ is a positive constant. If $N$ is an integer greater than $p$, then for any $\alpha\in \Nat^n\setminus \{0\}$ such that $|\alpha|\leq p$ the derivative $D^\alpha (f^N)$ is a linear combination with integral coefficients independent of $f$ of products
$$
f^{N-k}(D^{\beta_1}f)\dots (D^{\beta_k}f),
$$
where $k\in \{1,\dots, p\}$, \, $\beta_1,\dots, \beta_k\in \Nat^n\setminus \{0\}$ and $\beta_1 + \dots + \beta_k = \alpha$. It follows from the {\L}ojasiewicz inequality and the Hestenes lemma that for $N$ sufficiently big the function
$$
h(x) := \begin{cases} f^N(x), &\text{when $x\in \Real^n\setminus W$}\\ 0, &\text{when $x\in W$}\end{cases}
$$
is the required function.
\end{proof}

\end{document}